\documentclass[11pt]{article}
\usepackage{amsmath,amsthm}
\usepackage{amssymb,mathrsfs}
\usepackage{bm,mathtools}
\usepackage{graphics}
\usepackage{tikz}
\usepackage{float}
\usepackage{enumitem}
\usepackage{array}
\usetikzlibrary{calc}
\usepackage{xcolor}

\usepackage[normalem]{ulem}

\usepackage{geometry}
\usepackage[colorlinks=true,
linkcolor=blue,citecolor=blue,
urlcolor=blue]{hyperref}

\makeatletter
\def\@seccntDot{.}
\def\@seccntformat#1{\csname the#1\endcsname\@seccntDot\hskip 0.5em}
\renewcommand\section{\@startsection{section}{1}{\z@}%
	{18\p@ \@plus 6\p@ \@minus 3\p@}%
	{9\p@ \@plus 6\p@ \@minus 3\p@}%
	{\LARGE\bfseries\boldmath}}
\renewcommand\subsection{\@startsection{subsection}{2}{\z@}%
	{12\p@ \@plus 6\p@ \@minus 3\p@}%
	{3\p@ \@plus 6\p@ \@minus 3\p@}%
	{\bfseries\boldmath}}
\renewcommand\subsubsection{\@startsection{subsubsection}{3}{\z@}%
	{12\p@ \@plus 6\p@ \@minus 3\p@}%
	{\p@}%
	{\bfseries\boldmath}}
\makeatother

\renewenvironment{proof}{\noindent\textbf{Proof.}}{\hfill $\blacksquare$\par}

\usepackage{microtype}
\usepackage{float}

\theoremstyle{plain}

\newtheorem{theorem}{Theorem}

\newtheorem{lemma}[theorem]{Lemma}
\newtheorem{corollary}[theorem]{\rm\bfseries Corollary}

\newtheorem{observation}[theorem]{Observation}

\allowdisplaybreaks
\begin{document}
\title{The maximum number of cycles of a given length in a nonhamiltonian graph}
\author{
	Jifu Lin\thanks{Department of Mathematics, East China Normal University, Shanghai 200241, China, e-mail: {\tt jifulin01@163.com}.}
	\and 
	Xiaolin Wang\thanks{School of Mathematics and Statistics, Fuzhou University, Fuzhou 350108, China, e-mail: {\tt  xiaolinw@fzu.edu.cn}.}
	\and
	Lihua You\thanks{Corresponding author. School of Mathematical Sciences, South China Normal University, Guangzhou 510631, China, e-mail: {\tt ylhua@scnu.edu.cn}.}
}
\maketitle

\begin{abstract}
	In 2026, Li and Zhan characterized the nonhamiltonian graphs of order $n$ with the maximum number of paths of length $k$, where $n$ and $k$ are integers satisfying $1\leq k\leq n-1$. This work solves and generalizes a problem proposed by Erd\H{o}s in 1980. In this paper, we further determine the nonhamiltonian graphs of order $n$ attaining the maximum number of cycles of length $k$ for given integers $n$ and $k$ with $3\leq k\leq n-1$. As a corollary, we determine the generalized Turán number $\mathrm{ex}(n,C_k,C_n)$ for every $3\leq k\leq  n-1$.

\end{abstract}

{\bf Keywords.} Cycles; cycle lengths; nonhamiltonian graph; generalized Turán number; extremal graph

{\bf Mathematics Subject Classification.} 05C30, 05C35, 05C38
\vskip 8mm

\section{Introduction}
The enumeration of cycles, paths and other subgraphs in special graph classes is a fundamental problem in extremal graph theory. Ore \cite{Ore} and Bondy \cite{Bondy} proved the following classical result on the maximum size of a nonhamiltonian graph of a given order.
\begin{samepage}
\begin{theorem}[Ore \cite{Ore}, Bondy \cite{Bondy}]\label{ta}
	The maximum size of a nonhamiltonian graph of order $n \ge 6$ is $(n^2 - 3n + 4)/2$, and this maximum is attained uniquely by the graph $K_{n-1} \cdot K_2$, where $K_{n-1} \cdot K_2$ is the graph of order $n$ obtained by identifying a vertex of $K_{n-1}$ with a vertex of $K_2$.
\end{theorem}
\end{samepage}

In 1980, Erd\H{o}s (see \cite{Erd}, p.201) posed the following problem: For every positive integer $n$ and determine the nonhamiltonian graph of order $n$ attaining the maximum number of Hamilton paths. In 2026, Li and Zhan \cite{LZ} resolved a more general problem: for all integers $n,k$ satisfying $1 \le k \le n-1$, $K_{n-1} \cdot K_2$ is the unique nonhamiltonian graph maximizing the number of paths of length $k$. The case $k = n-1$ provides a solution to Erd\H{o}s's problem, whereas the case $k = 1$ corresponds to the classical result of Ore and Bondy. 

It is natural to ask which nonhamiltonian graphs of order $n$ maximize the number of cycles of length $k$, where $3 \le k \le n-1$. This is the problem addressed in this paper; equivalently, we determine the exact value of the generalized Turán number $\mathrm{ex}(n, C_k, C_n),$ where $\mathrm{ex}(n, T, H)$ denotes the maximum number of copies of $T$ in $n$-vertex $H$-free graphs. 

In a classical early result, Erdős (1962) \cite{Erd2} fully characterized the generalized Turán number $\mathrm{ex}(n, K_t, K_r)$ for any pair of integers satisfying $t < r$. Further extensions of this work can be found in \cite{BB}. In 2008, Bollob\'as and Gy\H{o}ri \cite{BG2008} found the asymptotic order of $\operatorname{ex}(n,C_3,C_5)$. Gy\H{o}ri and Li \cite{GL2012} further extended the bounds to general odd cycles $C_{2\ell+1}$, pioneering the study of cycle-type generalized Turán numbers. Following the work of Bollob\'as--Gy\H{o}ri--Li \cite{BG2008,GL2012}, Alon and Shikhelman \cite{Alon} established the theoretical framework for generalized Turán numbers $\operatorname{ex}(n,T,H)$ and derived results for cycles, cliques and trees.

Building on previous results of Bollob\'as, Gy\H{o}ri, and Li \cite{BG2008,GL2012} and of Alon and Shikhelman \cite{Alon}, Gishboliner and Shapira \cite{GS} obtained asymptotically tight bounds for $\operatorname{ex}(n,C_k,C_\ell)$ for all fixed integers $k$ and $\ell$. A number of papers have investigated the enumeration of paths and cycles of specified lengths in special families of graphs \cite{ref1,DG,EG,ref6,ref9}.

\noindent\textbf{Notation.} For integers $n$ and $k$, denote the permutation number by
$$
P(n, k) =
\begin{cases} 
	\dfrac{n!}{(n-k)!} & \text{if } 0 \leq k \leq n,\text{ and }  n\geq1, \\
	0 & \text{if } 1\leq n<k,\text{ or } k<0.
\end{cases}
$$

	Let $G_1$ and $G_2$ be vertex-disjoint graphs. The \emph{union} $G_1\cup G_2$ is the graph with vertex set $V(G_1)\cup V(G_2)$ and edge set $E(G_1)\cup E(G_2)$. For any positive integer $t$, let $tG$ denote the disjoint union of $t$ copies of $G$. The \emph{join} $G_1 \vee G_2$ is obtained from $G_1 \cup G_2$ by joining every vertex of $G_1$ with every vertex of $G_2$ by an edge. The main result of this paper is stated below:
\begin{samepage}
	\begin{theorem}\label{t1}
		Let $n$ and $k$ be integers with $3 \leq k \leq n - 1$. Then the maximum number of cycles of length $k$ in a nonhamiltonian graph of order $n$ is
		$$
		\frac{P(n - 1, k)}{2k}.
		$$
		The complete list of extremal graphs is $K_{n-1}\cdot K_2$ and $K_{n-1}\cup K_1$ when $(n,k)\neq(5,4)$, and $K_{4}\cdot K_2$, $K_{4}\cup K_1$, $K_{2,3}$, and $K_2\vee 3K_1$ when $(n,k)=(5,4)$.
	\end{theorem}
\end{samepage}

In 2018, Füredi, Kostochka and Luo \cite[Corollary 5]{FKL} obtained that $\operatorname{ex}(n, C_k, C_n) = P(n-1, k)/(2k)$ for $3 \leq k \leq (n-3)/9$. By Theorem~\ref{t1}, we immediately deduce the following corollary, which extends their result and further characterizes all extremal graphs.

\begin{corollary}
	For $3\leq k\leq n-1$, we have $\mathrm{ex}(n, C_k, C_n)=P(n-1, k)/(2k)$. The complete list of extremal graphs is $K_{n-1}\cdot K_2$ and $K_{n-1}\cup K_1$ when $(n,k)\neq(5,4)$, and $K_{4}\cdot K_2$, $K_{4}\cup K_1$, $K_{2,3}$, and $K_2\vee 3K_1$ when $(n,k)=(5,4)$.
\end{corollary}

 The paper is organized as follows. In Section 2, we introduce the notation and tools required for later proofs. In Section 3, we give the proof of Theorem \ref{t1}.
		
\section{Preliminaries}		
		
		Throughout this paper, we consider finite simple graphs and use standard terminology and notation from \cite{2,7}. Let $G$ be a graph with vertex set $V(G)$ and edge set $E(G)$, where $n=|V(G)|$ and $e(G)=|E(G)|$ are called the \emph{order} and \emph{size} of $G$, respectively. For $v \in V(G)$, let $N_G(v)$ and $\deg_G(v)$ denote the neighborhood and degree of $v$. 
		
		A \emph{cut-vertex} of a graph $G$ is a vertex whose deletion increases the number of connected components of $G$. A \emph{block} of $G$ is a maximal connected subgraph that contains no cut-vertex. For a subset $S \subseteq V(G)$, we denote by $G[S]$ the subgraph of $G$ induced by $S$, and by $G - S$ the subgraph obtained from $G$ by deleting the vertices in $S$ and their incident edges. We write $\overline{G}$ for the complement of $G$.  	
		
		Let $K_n$ be the complete graph on $n$ vertices, $K_{m,n}$ be the complete bipartite graph with vertex partitions of order $m$ and $n$, and $C_k$ be a cycle of length $k$ (or a $k$-cycle). A nonhamiltonian graph $G$ is said to be \textit{maximally nonhamiltonian} if for every non-edge $e \in E(\overline{G})$, the graph $G + e$ is hamiltonian. For maximally nonhamiltonian graphs, Li and Zhan \cite{LZ} gave the following useful results.
		
		\begin{lemma}[Li and Zhan \cite{LZ}]\label{LZ1}
			Let $G$ be a maximally nonhamiltonian graph of order $n \geq 3$. Then, for every non-edge $uv \in E(\overline{G})$, we have $\deg_{\overline{G}}(u) + \deg_{\overline{G}}(v) \geq n - 1$.
		\end{lemma}

		\begin{lemma}[Li and Zhan \cite{LZ}]\label{LZ2}
			The maximum number of Hamilton paths in a nonhamiltonian graph of order $n (\geq 6)$ is $(n-2)!$, and this maximum number is attained uniquely by $K_{n-1} \cdot K_2$.
		\end{lemma}
		
			A classical result of Erd\H{o}s and Gallai \cite{PE2} gives a sharp size threshold that ensures a graph to contain a cycle longer than a given length, with the exceptional extremal graphs characterized. Kopylov \cite{GK} sharpened the Erd\H{o}s--Gallai bound for $2$-connected graphs; Luo \cite{Luo} later generalized Kopylov's theorem to a clique-counting version. For $n\geq c\geq 4$ and $\frac c2>a\geq 1$, let $$H_{n,c,a}=K_a\vee (K_{c-2a}\cup (n-c+a)K_1)\text{,\quad}g(n,c,a)=e(H_{n,c,a}) = \binom{c-a}{2} + (n-c+a)a.$$ Clearly, when $a \ge 2$, $H_{n,c,a}$ is $2$-connected and has no cycle of length $c$ or longer.	
		\begin{lemma}[Kopylov \cite{GK}, Luo \cite{Luo}]\label{Ko}
			Let $n \ge c \ge 5$ and let $a = \left\lfloor \frac{c-1}{2} \right\rfloor$. If $G$ is a 2-connected $n$-vertex graph with circumference less than $c$, then $
			e(G) \le \max \{ g(n,c,2), g(n,c,a) \},$
			with equality only if $G = H_{n,c,2}$ or $G = H_{n,c,a}$.
		\end{lemma}

\section{Proof of Theorem \ref{t1}}\label{sec:3}
For an integer $k$ with $3 \leq k \leq n-1$, let $c_k(G)$ be the number of $k$-cycles in $G$, and let $G$ be an $n$-vertex nonhamiltonian graph with $3\leq k\leq n-1$. 

If $G \cong K_{n-1} \cdot K_2$ or $K_{n-1} \cup K_1$ when $(n,k) \neq (5,4)$, or $G \cong K_4 \cdot K_2$, $K_4 \cup K_1$, $K_{2,3}$, or $K_2 \vee 3K_1$ when $(n,k) = (5,4)$, then $c_k(G) = c_k(K_{n-1}) = P(n-1,k)/(2k)$. It suffices to prove $c_k(G) \leq P(n-1,k)/(2k)$, with equality only if $G \cong K_{n-1} \cdot K_2$ or $K_{n-1} \cup K_1$ for $(n,k) \neq (5,4)$, and $G \cong K_4 \cdot K_2$, $K_4 \cup K_1$, $K_{2,3}$, or $K_2 \vee 3K_1$ for $(n,k) = (5,4)$.

	In this section, to prove Theorem \ref{t1}, we distinguish two cases according to whether $G$ is 2-connected.  
\vspace{8pt}

{\noindent\bf Case 1.} $G$ is not 2-connected.

\vspace{8pt}
By the definition of $P(n,k)$, we have the following observation.

\begin{observation}\label{obs}
$P(n,k)$ is non-decreasing in $n$, and it is strictly increasing whenever $k\ge 1$ and $n\ge k-1$.
\end{observation}

\begin{lemma}\label{lemP1}
	Let $a,b,k$ be integers with $a,b\geq 1$ and $k\geq 3$. Then\\ {\rm(i)} $P(a+1,k)+P(b+1,k)\leq P(a+b,k)$ with equality if and only if  $a+b<k$ or $1\in\{a,b\}$.\\ {\rm (ii)} $P(a,k)+P(b,k)\leq P(a+b-1,k)$ with equality if and only if $a+b-1<k$ or $1\in\{a,b\}$.
\end{lemma} 

\begin{proof}
	(i)	If $a+b<k$ or $1\in\{a,b\}$, the equality holds since $a,b\geq 1$ and $k\geq 3$. Now we suppose that $a+b\geq k$ and $ a,b\ge 2$. Clearly, we have $P(i+1,k)-P(i,k)=kP(i,k-1)$ for $i\geq 1$. Combining Observation \ref{obs}, $a\geq 2$ and $k\geq 3$, we have
	\begin{align*}
		P(a+b,k)-P(a+1,k)&=\sum_{j=a+1}^{a+b-1}(P(j+1,k)-P(j,k))=\sum_{j=a+1}^{a+b-1}kP(j,k-1)\\&> \sum_{t=2}^{b}kP(t,k-1)= \sum_{t=2}^{b}(P(t+1,k)-P(t,k))\\&= P(b+1,k)-P(2,k)=P(b+1,k).
	\end{align*}
	 Hence (i) holds. 
	
	(ii) If $a+b-1<k$ or $1\in\{a,b\}$, then the equality holds since $a,b\geq 1$ and $k\geq 3$. Now we suppose that $a+b-1\geq k$ and $a,b\geq 2$. By (i) and Observation \ref{obs}, we have $P(a,k)+P(b,k)\leq P(a+b-2,k)<P(a+b-1,k)$. Hence (ii) holds.
\end{proof}
	\begin{samepage}
		\begin{lemma}\label{lemP3}
			Let $k \ge 3$ be an integer, and let $\{n_i\}_{i=1}^t$ be a sequence with $t \ge 2$ and $n_i \ge 1$ for each $i$. Then we have\\
			{\rm (i)} If $\sum_{i=1}^tn_i\geq k$, then $\sum_{i=1}^t P(n_i+1,k)\leq P\left(\sum_{i=1}^t n_i,k\right)$ with equality if and only if $t=2$ and $1\in\{n_1,n_2\}$. \\
			{\rm (ii)} If $\sum_{i=1}^tn_i\geq k+1$, then   $\sum_{i=1}^t P(n_i,k)\leq P\left(\sum_{i=1}^t n_i-1,k\right)$ with equality if and only if $t=2$ and $1\in\{n_1,n_2\}$.
		\end{lemma}
	\end{samepage}

	\begin{proof}
		If $t=2$, then the results of (i) and (ii) hold by Lemma \ref{lemP1}. Now we may assume $t\geq3$.
			
		(i) By repeatedly applying (i) of Lemma \ref{lemP1}, we obtain $$\sum_{i=1}^t P(n_i+1,k)\leq P(n_1+n_2,k)+\sum_{i=3}^{t} P(n_i+1,k)\leq P\left(\sum_{i=1}^{t}n_i-t+2,k\right)<P\left(\sum_{i=1}^t n_i,k\right),$$
		where the last inequality follows from $t\geq 3,\ \sum_{i=1}^t n_i\geq k$ and Observation \ref{obs}.
		
		(ii) By repeatedly applying (ii) of Lemma \ref{lemP1}, we obtain $$\sum_{i=1}^t P(n_i,k)\leq P(n_1+n_2-1,k)+\sum_{i=3}^{t} P(n_i,k)\leq P\left(\sum_{i=1}^{t}n_i-t+1,k\right)<P\left(\sum_{i=1}^t n_i-1,k\right),$$
		where the last inequality follows from $t\geq 3,\ \sum_{i=1}^t n_i\geq k+1$ and Observation \ref{obs}.
	\end{proof}

\vspace{5pt}

	We now consider the two subcases: $G$ is disconnected, and $G$ has connectivity $1$.
	
	If $G$ is disconnected, let its components have orders $n_1, \dots, n_t$, where $t \ge 2$ and $\sum_{i=1}^t n_i = n\geq k+1$. Since every cycle in a single component, (ii) of Lemma~\ref{lemP3} gives 
	\[
	c_k(G) \le \sum_{i=1}^t c_k(K_{n_i}) = \frac{1}{2k} \sum_{i=1}^{t} P(n_i, k)\leq \frac{P\left(n-1,k\right)}{2k} .
	\] If $c_k(G) =P\left(n-1,k\right)/(2k)$, then all the above inequalities become equalities. By (ii) of Lemma \ref{lemP3}, the last equality holds if and only if $t=2$ and $1\in\{n_1,n_2\}$. Thus, $G$ is isomorphic to a spanning subgraph of $K_{n-1}\cup K_1$, and furthermore, $G\cong K_{n-1}\cup K_1$ since $c_k(G) = P\left(n-1,k\right)/(2k)=c_k(K_{n-1})$.
	
	If $G$ has connectivity 1, then $G$ has at least two blocks, and every cycle lies entirely within some block. Let $b_1,\dots,b_t$ be the orders of the blocks of $G$, with $t\ge 2$ and each $b_i\ge 2$. Clearly, we obtain $\sum_i^t (b_i - 1) = n - 1\geq k.$
	Thus, by (i) of Lemma \ref{lemP3}, we have
	\[
	c_k(G) \leq \sum_{i=1}^t c_k(K_{b_i}) = \frac{1}{2k} \sum_{i=1}^t P(b_i, k) \leq \frac{P(n-1, k)}{2k}.
	\]
	If $c_k(G) = P\left(n-1,k\right)/(2k)$, then all the above inequalities become equalities. By (i) of Lemma \ref{lemP3}, the last equality holds if and only if $t=2$ and $2\in\{b_1,b_2\}$. Thus, $G$ is isomorphic to a spanning subgraph of $K_{n-1}\cdot K_2$, and furthermore, $G\cong K_{n-1}\cdot K_2$ since $c_k(G) = P\left(n-1,k\right)/(2k)=c_k(K_{n-1})$. 
	
	Combining the above arguments, we complete the proof of Case 1.
	
	\vspace{8pt}

{\noindent\bf Case 2.} $G$ is 2-connected.

\vspace{8pt}

If $n=4$, by Dirac's theorem, every $2$-connected graph on $4$ vertices is hamiltonian, a contradiction.
	
If $n=5$, then $\delta(G)\leq 2$; otherwise Dirac's theorem would imply that $G$ is hamiltonian. Since $G$ is $2$-connected, $\delta(G)=2$. Let $v$ be a vertex of degree $2$ with neighbors $u$ and $w$. Then $G-v$ is connected, and there is no Hamilton path from $u$ to $w$ in $G-v$ since $G$ is nonhamiltonian. After an easy check, $G\cong K_{2,3}$ or $G\cong K_2\vee 3K_1$. By direct verification, we have $c_{k}(G)\leq P(n-1,k)/(2k)$ with equality if and only if $k=4$ and $G$ is isomorphic to $K_{2,3}$ or $K_2\vee3K_1$. 

Now we assume that $n\geq 6$. Before proving, we need the following lemmas.

\begin{lemma}\label{lem2}
	Let $n\geq 6$, and let $G$ be an $n$-vertex nonhamiltonian 2-connected graph. Then $$c_{n-1}(G)< \frac{(n-2)!}{2}.$$
\end{lemma}

\begin{proof}
	Let $h(G)$ be the number of Hamilton paths of $G$. By Lemma \ref{LZ2}, since $G$ is nonhamiltonian, we have $h(G)\leq (n-2)!$ with equality if and only if $G\cong K_{n-1}\cdot K_2$. Moreover, since $G$ is 2-connected, $G\ncong K_{n-1}\cdot K_2$, and thus $h(G)<(n-2)!$.
	
	If $G$ has no $(n-1)$-cycle, the conclusion holds. Suppose now that $G$ has an $(n-1)$-cycle. Denote by $C$ an $(n-1)$-cycle of $G$, and let $\{v\} = V(G)\setminus V(C)$. Since $G$ is 2-connected, $\deg_G(v)\geq 2$. Let $u_1,u_2\in N_G(v)$. Clearly, $u_1$ and $u_2$ are not consecutive vertices in $C$ since $G$ is nonhamiltonian. 
	
	For $i\in\{1,2\}$, deleting one of the two edges incident to $u_i$ on $C$ and adding the edge $vu_i$ yields a Hamilton path. Thus each $(n-1)$-cycle produces at least 4 Hamilton paths.
	
	Conversely, each Hamilton path produced in this way can be obtained from at most two such $(n-1)$-cycles, since the vertex outside the cycle must be an endpoint of the path. Hence
	\[
	4c_{n-1}(G) \leq 2h(G),
	\]
	that is,
	\[
	c_{n-1}(G) \leq \frac{h(G)}{2} < \frac{(n-2)!}{2}.
	\] This proves the lemma.
\end{proof}

\begin{lemma}\label{lemsize}
	If $n\geq 6$, $H$ is an $n$-vertex nonhamiltonian 2-connected graph and $m=e(\overline{H})$, then $$m\geq \frac{(n-1)(3n-5)}{2(n+1)}$$ with equality only if $H\cong K_3\vee 4K_1$.
\end{lemma}

\begin{proof}
	By Lemma \ref{Ko}, since $H$ is an $n$-vertex nonhamiltonian 2-connected graph, we obtain 	\begin{equation}\label{eqg}
	e(H)\leq \max\{ g(n,n,2),g(n,n,a)\}=\max\left\{ \binom{n-2}{2} + 4,\ \binom{n-a}{2} + a^2 \right\},
	\end{equation} where equality holds only if $H\cong H_{n,n,2}$ or $H\cong H_{n,n,a}$, where $a=\lfloor(n-1)/2\rfloor$. 
	
	If $n\geq 10$, then $\binom{n-2}{2} + 4\geq\binom{n-a}{2} + a^2$. Thus, by (\ref{eqg}), $m=\binom{n}{2}-e(H)\geq \binom{n}{2}-\left(\binom{n-2}{2}+4\right)=2n-7>\frac{(n-1)(3n-5)}{2(n+1)}$. 
	
	If $6\leq n\leq 9$, then by (\ref{eqg}) and a direct computation, we have $$\begin{aligned}m= \binom{n}{2}-e(H)&\geq \begin{cases}
		n-1, & \text{if } n=6,7\\
		n+1,	& \text{if } n=8,9
	\end{cases}\\&\geq \dfrac{(n-1)(3n-5)}{2(n+1)},\end{aligned}$$ where all equalities hold only if $n=7$ and $e(H)=\binom{4}{2}+3^2$ (i.e., $H\cong H_{7,7,3}\cong K_3\vee 4K_1$).
\end{proof}

\begin{lemma}\label{lemcycle}
	Let $n$ and $k$ be integers with $n\geq 5$ and $3\leq k\leq n$. \\
	{\rm (i)} For a given edge $uv$, there exist exactly $P(n-2,k-2)$ $k$-cycles containing $uv$ in $K_n$.\\	
	{\rm (ii)} For two given adjacent edges $uv$ and $vw$, there exist exactly $P(n-3,k-3)$ $k$-cycles containing $uv$ and $vw$ in $K_n$. \\
	{\rm (iii)} For two given nonadjacent edges $uv$ and $xy$, there exist exactly $2(k-3)P(n-4,k-4)$ $k$-cycles containing both $uv$ and $xy$ in $K_n$.
\end{lemma}

\begin{proof}
	(i) Let $uv$ be an edge in $K_n$. To construct a $k$-cycle containing $uv$, we first choose and permute the other $k-2$ vertices $
	x_1, x_2, \dots, x_{k-2},$
	and then form a $k$-cycle.
	There are $P(n-2, k-2)$ ways to choose and order the $k-2$ vertices from the remaining $n-2$ vertices. Hence (i) holds.
	
	(ii) By an argument similar to that in (i), (ii) also holds.
	
	(iii) A triangle contains merely three edges, so no two disjoint edges can lie in a triangle when $k=3$. If $k=3$, the conclusion is trivial since $2(k-3)P(n-4,k-4)=0$. Now we may assume $k\geq4$. Let $uv$ and $xy$ be two nonadjacent edges, and fix the edge $uv$. We treat the edge $xy$ as an equivalent new vertex $z$. First, we select $k-4$ vertices from the remaining $n-4$ vertices of $K_n$ (excluding $u,v,x,y$), the number of ways is $\binom{n-4}{k-4}=\frac{(n-4)!}{(k-4)!(n-k)!}$.
	Combining vertex $z$ with the selected $k-4$ vertices yields $k-3$ elements, with $(k-3)!$ permutations.
	Moreover, the edge $xy$ has two possible orientations: $xy$ and $yx$, so we need to multiply the result by $2$. Therefore, the total number of distinct $k$-cycles in $K_n$ containing both nonadjacent
	edges $uv$ and $xy$ is $
	2(k-3)P(n-4,k-4).$
\end{proof}

\begin{lemma}\label{lemc5}
		$c_5(K_3\vee 4K_1)=36$.
\end{lemma}

\begin{proof}
	Let $H=K_3\vee 4K_1$. Let $A=\{a_1,a_2,a_3\}$ and $B=\{b_1,b_2,b_3,b_4\}$ be the vertex sets of the $K_3$ and $4K_1$ parts of $H$, respectively. For any $5$-cycle $C$ of $H$, since $B$ is an independent set of $H$, we have $|V(C)\cap A|=3$ and $|V(C)\cap B|=2$. Choosing two vertices $b_i$ and $b_j$ from $B$ yields $\binom{4}{2}$ possible selections. Then $H_{ij}=H[\{b_i,b_j,a_1,a_2,a_3\}]$ is $K_5$ minus an edge, and thus $c_5(H_{ij})=P(5,5)/10-P(3,3)=6$ by (i) of Lemma \ref{lemcycle}. Therefore, $c_5(H)=6\binom{4}{2}=36$.
\end{proof}

\vspace{8pt}

 By Lemma \ref{lem2}, it suffices to consider the case in which $G$ is 2-connected, $n\geq 6$ and $3\le k\le n-2$. We aim to prove that $c_k(G) < P(n-1, k)/(2k).$

	Let $H$ be an $n$-vertex maximally nonhamiltonian graph that contains $G$ as a spanning subgraph.	By Theorem \ref{ta}, since $H$ is nonhamiltonian and 2-connected, we have $e(H)< (n^2-3n+4)/2$. Then
	\begin{equation}\label{eqm}
		m=e({\overline{H}})=\frac{n(n-1)}{2}-e(H)>n-2.
	\end{equation} 
	
	Let $\mathcal{C}_k$ denote the set of all $k$-cycles in $K_n$. For $j\geq 0$, we define
	$$y_j=|\{C\in \mathcal{C}_k: |E(C)\cap E({\overline{H}})|=j\}|.$$ Let $$\Omega = \{(C,e) \mid C \in \mathcal{C}_k,\ e \in E(C) \cap E({\overline{H}})\},$$ $$\mathcal{P} = \{(C, \{e,f\}) \mid C \in \mathcal{C}_k,\ e,f \in E(C) \cap E({\overline{H}}),\ e \neq f\}.	$$
	Counting $\Omega$ by first fixing the cycle $C$ gives $|\Omega| = \sum_{j \geq 1} j y_j$. Counting $\Omega$ by first fixing the edge $e \in E({\overline{H}})$ and then using (i) of Lemma \ref{lemcycle}, gives $|\Omega| = mP(n-2,k-2)=m(n-2)B$, where $B=P(n-3,k-3)$. Hence \begin{equation}\label{eq1}
		\sum_{j \geq 1} j y_j =m(n-2)B.
	\end{equation}

	Let $s$ denote the number of unordered pairs of adjacent edges in ${\overline{H}}$, and $q$ the number of unordered pairs of nonadjacent edges in ${\overline{H}}$. By (ii) and (iii) of Lemma \ref{lemcycle}, every pair of adjacent edges is contained in $B$ $k$-cycles, and every pair of nonadjacent edges is contained in $\rho B$ $k$-cycles, where $\rho=2(k-3)/(n-3)$. Similarly, by double counting $|\mathcal{P}|$, we obtain that
	\begin{equation}\label{eq2}
		|\mathcal{P}|=\sum_{j \geq 1} \binom{j}{2} y_j = sB + q\rho B.
	\end{equation}

   We distinguish two subcases according to whether $\rho>1$ or $0\le \rho\le 1$.

	{\noindent\bf Subcase 1.} $\rho>1$.
	
		By Lemma \ref{LZ1}, we have $\deg_{\overline{H}}(u)+\deg_{\overline{H}}(v)\geq n-1$ for every $uv\in E(\overline{H})$. Then we have
	$$
	\sum_{v\in V(\overline{H})} \deg_{\overline{H}}(v)^2
	=\sum_{uv\in E(\overline{H})}\big(\deg_{\overline{H}}(u)+\deg_{\overline{H}}(v)\big)
	\geq m(n-1).
	$$
	Thus, we obtain \begin{equation}\label{eq3}
		s=\sum_{v\in V(\overline{H})} \binom{\deg_{\overline{H}}(v)}{2}=\frac12\left(\sum_{v\in V(\overline{H})} \deg_{\overline{H}}(v)^2-2m\right)\geq \frac{m(n-3)}{2}.
	\end{equation}
	
	By (\ref{eq2}), (\ref{eq3}) and the fact $s+q=\binom{m}{2}$ and $\rho>1$, we have $$\sum_{j \geq 2} \binom{j}{2} y_j = sB + q\rho B=\left(\rho\binom{m}{2}-(\rho-1)s\right)B\leq \left(\rho\binom{m}{2}-(\rho-1)\frac{m(n-3)}{2}\right)B.$$ Combining (\ref{eq1}), we have 
	\begin{equation}
		\begin{aligned}
			\sum_{j\ge 1} j^2 y_j &= \sum_{j\ge 1} j y_j + 2\sum_{j\ge 2} \binom{j}{2} y_j \\
			&\le m(n-2)B + m\left[\rho(m-1) - (\rho-1)(n-3)\right]B= mLB,
		\end{aligned}
		\label{eq4}
	\end{equation}
	where $L=n-2+\rho(m-1)-(\rho-1)(n-3)=2n+1-2k+2(k-3)(m-1)/(n-3).$ Since $k\leq n-2$, this expression shows that $L>0$.
	
	By the Cauchy–Schwarz inequality, we obtain $$\left(\sum_{j\geq 1}jy_j\right)^2\leq \left(\sum_{j\geq 1}y_j\right)\left(\sum_{j\geq 1}j^2y_j\right).$$ Combining (\ref{eq1}) and (\ref{eq4}), we have \begin{equation}\label{eq6}
		\sum_{j\geq 1}y_j\geq \frac{(m(n-2)B)^2}{mLB}= \frac{m(n-2)^2}{L}B.
	\end{equation}
	
	We first show $\frac{m(n-2)^2 B}{L} \geq \frac{(n-1)(n-2)B}{2}$; that is, $2m(n-2)\geq (n-1)L$. It remains to verify $$f(n,m,k)=2m(n-2)(n-3)-(n-1)[(n-3)(2n+1-2k)+2(k-3)(m-1)]\geq 0.$$
	Since $k \leq n-2$, the coefficient of $m$ in $f(n,m,k)$ satisfies
	\[
	2\bigl((n-2)(n-3) - (n-1)(k-3)\bigr) \geq 2(n+1) > 0.
	\]
	Then $f(n,m,k)$ is increasing in $m$. Moreover, by (\ref{eqm}), we have $m \geq n-1$, so $f(n,m,k)$ is decreasing in $k$. Hence we have
	$f(n,m,k)\geq f(n,m,n-2) = 2(n+1)m - (n-1)(3n-5)\geq 0$ by Lemma \ref{lemsize}, where all equalities hold only if $k=n-2=5$ and $H\cong K_3\vee 4K_1$. Thus, by (\ref{eq6}), we have $$\sum_{j\geq 1}y_j\geq \frac{m(n-2)^2 B}{L} \geq \frac{(n-1)(n-2)B}{2}=\frac{P(n-1,k-1)}{2}.$$
	 
	By the above arguments, \begin{equation}\label{eq9}
	c_k(G)\leq c_k(H)= y_0=|\mathcal{C}_k|-\sum_{j\geq 1}y_j\leq \frac{P(n,k)}{2k}-\frac{P(n-1,k-1)}{2}=\frac{P(n-1,k)}{2k},
	\end{equation} where all equalities hold only if $f(n,m,k)=0$, which implies $k=5$ and $H\cong K_3\vee 4K_1$. If $c_k(G)=P(n-1,k)/(2k)$, then $k=5$ and $n=7$, so $\rho=2(k-3)/(n-3)=1$, contradicting $\rho>1$. Hence $c_k(G)<P(n-1,k)/(2k)$.
	
	{\noindent\bf Subcase 2.} $0\leq \rho  \leq 1$.
	
	In this case, we have $\rho B\leq B$. By (\ref{eq2}) and the fact $s+q=\binom{m}{2}$, we have \begin{equation}\label{eq10}
		\sum_{j \geq 1} \binom{j}{2} y_j = sB + q\rho B\leq (s+q)B=\binom{m}{2} B.
	\end{equation} Combining (\ref{eq1}) and (\ref{eq10}), we have 
	\begin{equation}
		\begin{aligned}
			\sum_{j\ge 1} j^2 y_j = \sum_{j\ge 1} j y_j + 2\sum_{j\ge 2} \binom{j}{2} y_j 
			\le m(n-2)B + 2\binom{m}{2} B= m(m+n-3)B,
		\end{aligned}
		\label{eq5}
	\end{equation} As in the derivation of (\ref{eq6}), by the Cauchy–Schwarz inequality, (\ref{eq1}) and (\ref{eq5}), we have \begin{equation}\label{eq7}
	\sum_{j\geq 1}y_j\geq \frac{(m(n-2)B)^2}{m(m+n-3)B}= \frac{m(n-2)^2}{m+n-3}B.
	\end{equation} By (\ref{eqm}), we have $m\geq n-1$, and since the function $x/(x+n-3)$ is increasing for $x>0$, we have \begin{equation}\label{eq8}
	\sum_{j\geq 1}y_j\geq \frac{m(n-2)^2}{m+n-3}B\geq \frac{(n-1)(n-2)^2}{2n-4}B= \frac{P(n-1,k-1)}{2}.
	\end{equation} Similar to (\ref{eq9}), we have $c_k(G)\leq P(n-1,k)/(2k)$, where equality holds only if the equalities in \eqref{eq10} and \eqref{eq8} both hold. 
	
	If $q=0$, then the edges in $\overline{H}$ are pairwise adjacent. Since $m\geq n-1\geq 5$ by (\ref{eqm}), $H$ cannot be a triangle; hence it is a star $K_{1,n-1}$, so $H=K_{n-1}\cup K_1$, contradicting the $2$-connectivity of $H$. Hence $q>0$, and equality in \eqref{eq10} forces $\rho=1$. 
	
	If $c_k(G)= P(n-1,k)/(2k)$, then by the above arguments, (\ref{eq10}) and (\ref{eq8}), we have $\rho=1$ and $m=n-1$. By Lemma \ref{lemsize}, we have $m=n-1\geq \frac{(n-1)(3n-5)}{2(n+1)}$ with equality only if $H\cong K_3\vee 4K_1$, and thus $n\leq 7$. Then by $n\geq 6$, we have $6\leq n\leq 7$. Since $\rho=1$, we have $(n,k)=(7,5)$. For $n=7$, $m=n-1$ equals the lower bound in Lemma \ref{lemsize}, so the equality case of that lemma
	gives $H\cong K_3\vee 4K_1$. However, by Lemma \ref{lemc5}, we have $c_5(G)\leq c_5(H)=36<72=P(6,5)/10$, a contradiction. Hence $c_k(G)<P(n-1,k)/(2k)$.
	
Combining the above arguments, we complete the proof of Case 2. Cases 1 and 2 together complete the proof of Theorem \ref{t1}. \hfill $\blacksquare$

\section*{Funding}
This work is supported by the National Natural Science Foundation of China (Grant Nos. 12371347, 12271337, 12401447).

\section*{Declaration}

\noindent\textbf{Conflict~of~interest}
The authors declare that they have no known competing financial interests or personal relationships that could have appeared to influence the work reported in this paper.

\noindent\textbf{Data~availability}
No data was used for the research described in the article.

\end{document}